\documentclass[final,1p,times]{elsarticle}
\usepackage{amssymb}
\usepackage{amsmath}
\usepackage{amsthm}
\usepackage[utf8]{inputenc}
\usepackage{marginnote}
\usepackage{amsfonts}
\usepackage{textcomp}
\usepackage{color,soul}
\usepackage{placeins}
\usepackage{tikz}
\usepackage{makecell}
\usepackage{amssymb}
\usepackage{graphicx}
\usepackage{hyperref}
\usepackage[none]{hyphenat}
\usepackage{lipsum}
\newtheorem{thm}{Theorem}[section]
\newtheorem{lem}[thm]{Lemma}
\newtheorem{rem}[thm]{Remark}
\newtheorem{prop}[thm]{Proposition}
\newtheorem{cor}[thm]{Corollary}
\newtheorem{example}[thm]{Example}
\newdefinition{defn}[thm]{Definition}
\allowdisplaybreaks

\journal{}
\begin{document}
	\begin{frontmatter}
	
\title{ Designing optimal dual frames for $\ell^p-$average error optimization}
        \author{Shankhadeep Mondal}
		\ead{shankhadeep.mondal@ucf.edu}
        \author{Deguang Han}
		\ead{deguang.han@ucf.edu}
		\author{R. N. Mohapatra}
		\ead{Ram.Mohapatra@ucf.edu}

		\address{Department of Mathematics, University of Central Florida}
		\cortext[]{}		
		
		\begin{abstract}
		In this paper, we investigates the problem of optimal dual frame selection for signal reconstruction in the presence of erasures. Unlike traditional approaches relying on left inverses, we evaluate performance through the norms of error operators, using the Frobenius norm, spectral radius, and numerical radius as measures. Our central focus is the characterization of dual frames that minimize the $\ell^p-$average under these error operator measurements over all possible erasure patterns. We provide conditions under which the canonical dual frame is uniquely optimal and extend our results to multiple erasures. In the Frobenius norm case, we offer a complete characterization for any number of erasures in uniform tight frames. The paper also examines interconnections between optimality criteria across different norm measures and gives sufficient conditions ensuring uniqueness of the optimal dual.

		\end{abstract}

		
	\end{frontmatter}
	
	\section{Introduction}
	
	The theory of frames has emerged as a powerful mathematical framework for robust signal representation, with wide-ranging applications in signal processing, wireless communications, and data transmission. One of the most compelling features of frames is their inherent redundancy, which provides resilience against the loss of information which is commonly referred to as erasures during transmission. When part of the encoded frame coefficients are lost due to adverse network conditions or channel imperfections, the redundant structure of frames allows for accurate reconstruction of the original signal, provided that an appropriate dual frame is chosen.

A fundamental problem in this context is the selection of an optimal dual frame that minimizes the reconstruction error introduced by such erasures. Various norms of the associated error operators have been employed to quantify this reconstruction error \cite{peh, shan4, jerr, jins}. Classical works have addressed the operator norm-based optimization of dual frames, notably for Parseval frames or for specific classes of erasures\cite{holm}. However, a comprehensive treatment of different operator-induced measurements in the broader setting of general frames remains relatively unexplored.

In recent years, extensive research has been devoted to addressing erasure problems through the lens of frame theory. Casazza and Kovačević \cite{casa2} explored equal norm tight frames in depth, analyzing their structure, construction, and robustness against erasures. Goyal, Kovačević, and Kelner \cite{goya} examined uniform tight frames from a coding-theoretic perspective, showing that they are optimal for handling single erasures. Holmes and Paulsen \cite{holm} introduced an operator norm-based framework to assess the optimality of dual frames, particularly focusing on Parseval frames and their canonical duals, and identified conditions under which the canonical dual minimizes the worst-case reconstruction error. This line of work was extended in \cite{jerr, jins}, where the uniqueness and topological properties of optimal duals were studied, especially under multiple erasure settings.

Alternative approaches have also emerged. Bodmann and Paulsen \cite{bodm1} proposed an average-case measure for reconstruction error based on operator norms, while Pehlivan, Han, and Mohapatra \cite{sali} investigated the spectral radius as an error metric for one erasure, with further developments for two erasures in \cite{peh, dev}. The numerical radius was analyzed in \cite{ara}, and an averaged combination of operator norm and numerical radius in \cite{deep}. While our work adopts a similar average-type measurement as in \cite{bodm1}, it significantly generalizes the setting by considering arbitrary frames and their duals, rather than restricting to Parseval frames and its canonical duals. Related investigations using probabilistic models and weight sequences were conducted in \cite{leng3, li1}, while broader studies on optimal dual pairs, including Parseval and general frames can be found in \cite{leng, li, shan}. In recent years, several generalized frame structures have been developed, among which \(K-\)frames introduced by G\u{a}vru\c{t}a \cite{gua} have gained particular attention due to their role in reconstruction within \(\mathrm{Range}(K)\). Structural and constructive results on \(K-\)frames and their duals appear in \cite{du,he,li4,bar,wang,nami,li3,pill,xiao,koba}, while erasure problems for \(K-\)frames are addressed in \cite{miao,he}. These studies lay the groundwork for further developments in designing optimal duals in the presence of erasure and reconstruction challenges.

\par In this paper, we investigate the problem of optimal dual frame selection for a given finite frame, under erasure-induced error, using a measurement that is structurally different from the traditional min–max formulations in the literature. More precisely, we consider the 
$\ell^p$-average of the norm of the error operators over all possible $m-$element erasure patterns. This formulation offers a natural probabilistic interpretation and avoids the over-conservatism of worst-case models. For instance,  rather than minimizing the \emph{worst-case error measure}
\[
\max \bigg\{ \mathcal{M}\big( T_G^* D T_F \big) : D \in \mathcal{D}_m \bigg\},
\]
where $\mathcal{M}$ denotes a chosen measure of the error operator and 
$\mathcal{D}_m$ is the collection of diagonal matrices with exactly $m$ ones 
(indicating the retained coefficients) and $N-m$ zeros on the main diagonal, as considered by Bodmann and Paulsen~\cite{bodm1}, 
we instead focus on minimizing the \emph{$\ell^p$-average error}, given by
\[
\left( \frac{1}{\binom{N}{m}} \sum_{D \in \mathcal{D}_m} 
\big[ \mathcal{M} \big( T_G^* D T_F \big) \big]^p \right)^{\!\frac{1}{p}},
\]
for $p>1$. Where $G$ is a dual frame of $F$ and $\mu$ is one of the error operator measures, namely the Frobenius norm, spectral radius, or numerical radius. Unlike previous studies that often restrict to Parseval frames and their canonical duals, our framework applies to general frames and arbitrary duals, significantly broadening the scope of the analysis.

The structure of the paper is as follows: Section 2 introduces the necessary background on frames, duals, error operators, and matrix norms. Section 3 formalizes the average case measurement model and error formulation. Section 4 addresses the Frobenius norm and provides complete characterizations of optimal dual frames for any $m-$erasure scenario. Section 5 investigates the spectral radius as the error measure, and Section 6 focuses on the numerical radius. In Section 7, we study the relationships between optimal duals under the different norm measures and analyze their possible intersections and topological structure.

This work contributes to the ongoing development of robust frame theory by extending optimal dual frame selection beyond canonical settings and introducing average-based, norm-sensitive optimization criteria with both theoretical and practical significance.
	
	\section{Preliminaries on erasure }
	
Throughout this article, we consider $\mathcal{H}_n$ to be an $n$-dimensional Hilbert space. A finite collection of vectors $F = \{f_i\}_{i=1}^N$ in $\mathcal{H}_n$ is called a \textit{frame} for $\mathcal{H}_n$ if there exist constants $A, B > 0$ such that
\[
A' \|f\|^2 \leq \sum_{i=1}^N |\langle f, f_i \rangle|^2 \leq B' \|f\|^2, \quad \text{for all } f \in \mathcal{H}_n.
\]
The constants $A'$ and $B'$ are referred to as the \textit{lower} and \textit{upper frame bounds}, respectively. The \textit{optimal lower frame bound} is the supremum of all possible lower bounds, and the \textit{optimal upper frame bound} is the infimum over all upper bounds. A frame $\{f_i\}_{i=1}^N$ is called \textit{tight} if $A' = B'$, i.e., $\sum_{i=1}^N |\langle f, f_i \rangle|^2 = A' \|f\|^2, \quad \text{for all } f \in \mathcal{H}_n.$ In particular, if $A' = B' = 1$, then the frame is said to be a \textit{Parseval frame}.

 Given a frame $F$ for $\mathcal{H}_n,$ the \textit{analysis operator} $T_F : \mathcal{H}_n \to \mathbb{C}^N$ is defined by
\[
T_F(f) = \{\langle f, f_i \rangle\}_{i=1}^N,
\]
and its adjoint $T_F^* : \mathbb{C}^N \to \mathcal{H}_n$, called the \textit{synthesis} or \textit{pre-frame operator}, is given by
\[
T_F^*(\{c_i\}_{i=1}^N) = \sum_{i=1}^N c_i f_i.
\]
The \textit{frame operator} $S_F : \mathcal{H}_n \to \mathcal{H}_n$ is defined by
\[
S_F f = T_F^* T_F f = \sum_{i=1}^N \langle f, f_i \rangle f_i.
\]
It is well known that $S_F$ is a positive, self-adjoint, and invertible operator.

A sequence $G = \{g_i\}_{i=1}^N$ in $\mathcal{H}_n$ is called a \textit{dual frame} of $F$ if every $f \in \mathcal{H}_n$ admits the reconstruction
\begin{equation}\label{eqn2point1reconstruction}
    f = \sum_{i=1}^N \langle f, f_i \rangle g_i, \quad \text{for all } f \in \mathcal{H}_n,
\end{equation}
which implies that $T_G^* T_F = I$. This further leads to the identity
\[
f = \sum_{i=1}^N \langle f, g_i \rangle f_i, \quad \text{for all } f \in \mathcal{H}_n.
\]
Therefore, \( G \) is a dual of \( F \) if and only if \( F \) is a dual of \( G \). A pair \( (F, G) \), where each frame contains \( N \) elements, is referred to as an \textit{\((N, n)\)-dual pair} in the \( n \)-dimensional Hilbert space \( \mathcal{H}_n \). It is well known that for any frame \( F = \{f_i\}_{i=1}^N \), the collection \( \{S_F^{-1} f_i\}_{i=1}^N \) forms a dual frame, commonly called the \textit{canonical} or \textit{standard dual}. When \( F \) is a basis, this canonical dual is the only dual of $F$. In contrast, if \( F \) is not a basis, infinitely many dual frames exist. In fact, every dual frame \( G = \{g_i\}_{i=1}^N \) of \( F \) can be represented as  $g_i = S_F^{-1} f_i + u_i, \quad \text{where } \{u_i\}_{i=1}^N \subset \mathcal{H}_n,$ subject to the constraint  
\[
\sum_{i=1}^N \langle f, u_i \rangle f_i = \sum_{i=1}^N \langle f, f_i \rangle u_i = 0, \quad \forall\, f \in \mathcal{H}_n.
\]

\noindent
An additional property of a dual pair \( (F, G) \) is the relation  $\sum_{i=1}^N \langle f_i, g_i \rangle = \operatorname{tr}(T_G T_F^*) = \operatorname{tr}(T_F^* T_G) = \operatorname{tr}(I) = n.$
For a comprehensive treatment of frame theory, see \cite{ole}.

   The idea of viewing frames as coding mechanisms begins by considering a vector 
\( f \in \mathcal{H}_n \) and an \((N,n)\)-frame \( F \) with analysis operator \( T_F \). 
The vector \( T_F f \) serves as an encoded representation of \( f \), 
which is transmitted to a receiver and subsequently decoded using \( T_G^* \), 
where \( G \) is a dual frame of \( F \).  During transmission, some frame coefficients in the reconstruction process may be lost,  corrupted or delayed, requiring the reconstruction of \( f \) from the available coefficients. 
If certain coefficients are erased, the received signal can be modeled as
\[
T_G^* (I_{N \times N} - D) T_F f,
\]
where \( D \) is a diagonal matrix with \( m \) ones and \( N-m \) zeros, 
and the ones in \( D \) indicate the locations of the missing coefficients of 
\( T_G^* T_F f \). Assuming erasures occur in exactly \( m \) positions, let 
\( \Lambda \subseteq \{1,2,\dots,N\} \) with \( |\Lambda| = m \) denote the set of erased indices. 
The associated \emph{error operator} is then defined as
\[
E_{\Lambda} f := T_G^* D T_F f 
= \sum_{i \in \Lambda} \langle f, f_i \rangle g_i,
\]
where \( D \) is the diagonal matrix with entries 
\( d_{ii} = 1 \) for \( i \in \Lambda \) and \( d_{ii} = 0 \) otherwise. 
This operator quantifies the reconstruction error introduced by the erased components.

   To assess the error across all possible $m$-erasure patterns, we consider 
\[
\left\{\frac{1}{\binom{N}{m}}\sum\limits_{D \in \mathcal{D}_m }\left(\mathcal{M} \left(T_G^* D T_F \right)\right)\right\}^{\dfrac{1}{p}},
\]
as error measurement, where $\mathcal{M}$ denotes a suitable measure of the operator (e.g., spectral radius or operator norm), and $\mathcal{D}_m$ is the collection of all such diagonal matrices $D$ corresponding to erasures in exactly $m$ locations i.e., having $m$ entries equal to $1$ and $N - m$ entries equal to $0$ on the diagonal.

    \begin{defn}
For a frame \( F = \{f_i\}_{i=1}^N \) in \( \mathcal{H}_n \), 
a dual frame \( G = \{g_i\}_{i=1}^N \)  is said to be a \emph{\(1-\)uniform} dual of $F$
if \( \langle f_i, g_i \rangle \) are identical for all \( i \). 
Equivalently, there exists a constant \( c \in \mathbb{C} \) such that $\langle f_i, g_i \rangle = c \quad \text{for all } 1 \leq i \leq N.$
\end{defn}

\begin{rem}\label{rem3point2}
	It is worth noting that if $G$ is a $1-$uniform dual of $F$, then the constant inner product value $c$ turns out to be $\frac{n}{N}$. This follows from the relation $n = \sum_{i=1}^N \langle f_i, g_i \rangle = cN.$
\end{rem}

	\begin{defn}
Let \( F = \{f_i\}_{i=1}^N \) be a frame for \( \mathcal{H}_n \), and let \( G = \{g_i\}_{i=1}^N \) be a \(1-\)uniform dual of \( F \). 
We say that \( G \) is a \emph{\(2-\)uniform dual} of \( F \) if there exists a constant \( c' \in \mathbb{C} \) such that $\langle f_i, g_j \rangle \, \langle f_j, g_i \rangle = c'\quad \text{for all } 1 \leq i \neq j \leq N.$
\end{defn}
\vspace{0.5cm}

\begin{prop}\label{thm2point4minq}\cite{mapa}
Let \( a > 0 \) with \( a \neq 1 \), and let \( m \in \mathbb{Q} \). Then
\[
a^m - 1 
\begin{cases}
> m(a-1), & \text{if } m \notin (0,1), \\[6pt]
< m(a-1), & \text{if } m \in (0,1).
\end{cases}
\]
\end{prop}

 \begin{prop}\cite{ole} \label{olethm1}
  Let $\{f_k \}_{i=1}^N$ be a frame for $n$ dimensional Hilbert space $\mathcal{H}_n$. Let $ f \in \mathcal{H}_n$. If $f$ has a representation $ f = \sum\limits_{i=1}^N c_i f_i$, for some coefficiant $\{c_i\}$, then \\
  \begin{equation} \label{oleeqn}
     \sum_{i=1}^N |c_i|^2 = \sum_{i=1}^N |\langle f, S_F^{-1}f_i \rangle |^2 + \sum_{i=1}^N |c_i - \langle f, S_F^{-1}f_i \rangle |^2  .
       \end{equation}  
  
  \end{prop}
  
  \begin{prop}\label{}
  Let $G $ be a $1-$uniform dual of $F$. Then,
  $\sum\limits_{i \neq j} | \langle f_i , g_j \rangle |^2 \geq n- \frac{n^2}{N}$ . 
 
  \end{prop}
 \begin{proof}
   As $G$ is a $1-$uniform dual of $F$, we have $f_i= \sum\limits_{j=1}^N \left\langle f_i , g_j \right\rangle f_{j} ,\;\;\;\forall 1\leq i \leq N.$
   By Theorem \ref{olethm1}, we have
   $$ \sum_{j=1}^N |\langle f_i , g_j \rangle|^2 \geq \sum_{j=1}^N |\langle f_i , S_F^{-1}f_j \rangle|^2 ,\;\;\; 1 \leq i \leq N .$$
   Summing over $1 \leq i \leq N,$ we have,  
\begin{align} \label{eqn2point3}
& \sum_{i=1}^N \sum_{j=1}^N \bigg|\left\langle f_i , g_j \right\rangle \bigg|^2 \geq  \sum_{i=1}^N\sum_{j=1}^N \bigg|\left\langle f_i , S_F^{-1}f_j \right\rangle \bigg|^2   =  \sum_{i=1}^N\left( \sum_{j=1}^N \left\langle f_i , S_F^{-1}f_j \right\rangle \left\langle S_F^{-1}f_j , f_i  \right\rangle \right)\nonumber \\&= \sum_{i=1}^N\left( \left\langle  \sum_{j=1}^N \langle f_i , S_F^{-1}f_j \rangle f_j , S_F^{-1}f_i \right\rangle  \right)  = \sum_{i=1}^N \langle f_i , S_F^{-1}f_i \rangle  = n.
\end{align} 
 Rewriting the above inequality we obtain ,
\begin{equation}\label{eqn2point4}
\sum_{i=1}^N |\langle f_i , g_i \rangle |^2 + \sum_{i \neq j} |\langle f_i , g_j \rangle |^2 \geq n . 
\end{equation} 

 As $(F,G)$ is a $1-$uniform dual pair,  substituting $\langle f_i, g_i \rangle  = \frac{n}{N},\;1\leq i \leq N $ in \eqref{eqn2point4}
we obtain  $\sum\limits_{i \neq j} | \langle f_i , g_j \rangle |^2 \geq n- \frac{n^2}{N}$ .\\

  \end{proof}

\begin{thm}\label{lemma2point4}\cite{mondal}
Let \( a_1, a_2, \ldots, a_s \) be positive real numbers, and let \( q_1, q_2, \ldots, q_s \) be positive real weights. Then for any rational number \( m \in (0,1) \), 
\[
\frac{q_1 a_1^m + q_2 a_2^m + \cdots + q_s a_s^m}{q_1 + q_2 + \cdots + q_s}
\begin{cases}
\geq \left( \dfrac{q_1 a_1 + q_2 a_2 + \cdots + q_s a_s}{q_1 + q_2 + \cdots + q_s} \right)^m, & \text{if } m \notin (0,1), \\~\\
\leq \left( \dfrac{q_1 a_1 + q_2 a_2 + \cdots + q_s a_s}{q_1 + q_2 + \cdots + q_s} \right)^m, & \text{if } m \in (0,1),
\end{cases}
\]
with equality if and only if \( a_1 = a_2 = \cdots = a_s \).
\end{thm}

\section{Optimal dual frames under Frobenius norm}

In this section, we study the problem of identifying dual frames that minimize the average reconstruction error when a single erasure occurs, using the \textit{Frobenius norm} as the measurement of the associated error operator. The framework is then extended to the case of multiple erasures.

Let $F = \{f_i\}_{i=1}^N$ be a frame for a finite-dimensional Hilbert space $\mathcal{H}_n$, and let $G = \{g_i\}_{i=1}^N$ be a dual frame of $F$. For $p > 1$, we define the average Frobenius-norm error for a  $1-$erasure as
\begin{align*}
\mathrm{AE}_{\mathfrak{F}}^{(1),p}(F, G):= \left\{ \frac{1}{N} \sum\limits_{D \in \mathcal{D}^{(1)}} \left\| T_{G}^* D T_{F} \right\|_{\mathfrak{F}}^p \right\}^{1/p},
\end{align*}
where $\| \cdot \|_{\mathfrak{F}}$ denotes the Frobenius norm and $\mathcal{D}^{(1)}$ is the set of all diagonal matrices in $\mathbb{C}^{N \times N}$ with exactly one entry equal to one and the remaining entries zero. The corresponding optimal error value over all dual frames of $F$ is given by
\begin{align*}
\mathrm{AE}_{\mathfrak{F}}^{(1),p}(F) := \inf \left\{ \mathrm{AE}_{\mathfrak{F}}^{(1),p}(F, G) : G \text{ is a dual of } F \right\},
\end{align*}
and the collection of dual frames achieving this infimum is denoted by
\begin{align*}
\zeta_{\mathfrak{F}}^{(1),p}(F) := \left\{ G\,\text{be a dual of $F$} : \mathrm{AE}_{\mathfrak{F}}^{(1),p}(F, G) = \mathrm{AE}_{\mathfrak{F}}^{(1),p}(F) \right\}.
\end{align*}
Every element of \( \zeta_{\mathfrak{F}}^{(1),p}(F) \) is referred to as a \textit{$1$-erasure Frobenius-optimal dual frame of \( F \)}.

This framework naturally generalizes to $k-$erasure settings. For $k > 1$, define
\begin{align*}
	&\mathrm{AE}_{\mathfrak{F}}^{(k),p}(F, G) := \left\{ \frac{1}{\binom{N}{k}} \sum\limits_{D \in \mathcal{D}^{(k)}} \left\| T_{G}^* D T_{F} \right\|_{\mathfrak{F}}^p \right\}^{1/p}, \\&
	\mathrm{AE}_{\mathfrak{F}}^{(k),p}(F) := \inf \left\{ \mathrm{AE}_{\mathfrak{F}}^{(k),p}(F, G) : G \text{ is a $(k-1)-$erasure optimal dual of } F \right\}, \\&
	\zeta_{\mathfrak{F}}^{(k),p}(F)  := \left\{ G' \text{ is a dual of } F : \mathrm{AE}_{\mathfrak{F}}^{(k),p}(F, G) = \mathrm{AE}_{\mathfrak{F}}^{(k),p}(F) \right\}.
\end{align*}
Every element in $\zeta_{\mathfrak{F}}^{(k),p}(F)$ is referred to as a \textit{$k-$erasure optimal dual frame of $F$ under the Frobenius norm}.

\noindent When an erasure occurs at the \( i^{\text{th}} \) index, the associated error operator is given by  
\[
E_{\Lambda,F,G} f = \langle f, f_i \rangle g_i,
\]
where \( E_{\Lambda,F,G} \) represents the error contribution resulting from the loss of the \( i^{\text{th}} \) frame coefficient. It can be easily verified that the Frobenius norm of the error operator \( E_{\Lambda,F,G} \), corresponding to $1-$erasure at index \( i \), satisfies:
\begin{align}\label{eqn2point5}
\left\|E_{\Lambda,F,G}\right\|_{\mathfrak{F}} 
&= \left\|T_{G}^* D T_{F}\right\|_{\mathfrak{F}} \nonumber\\
&= \sqrt{\mathrm{tr}\left( T_{F}^* D T_{G} T_{G}^* D T_{F} \right)} \nonumber\\
&= \sqrt{\mathrm{tr}\left( D T_{F} T_{F}^* D T_{G} T_{G}^* \right)} \nonumber\\
&= \sqrt{ \|f_i\|^2 \cdot \|g_i\|^2 } \\&= \|f_i\| \cdot \|g_i\|.
\end{align}

This leads to a simplified expression for the average Frobenius-norm error for $1-$erasure:
\[
\mathrm{AE}_{\mathfrak{F}}^{(1),p}(F, G) = \left\{ \frac{1}{N} \sum_{i=1}^N \left( \|f_i\| \cdot \|g_i\| \right)^p \right\}^{1/p}.
\]

For any $(N, n)$-dual pair $(F, G)$ in $\mathcal{H}_n$, the following inequality holds:
\begin{align}\label{eqn4.6}
\mathrm{AE}_{\mathfrak{F}}^{(1),p}(F, G)
&= \left\{ \frac{1}{N} \sum_{i=1}^N \|f_i\|^p \cdot \|g_i\|^p \right\}^{1/p} \notag \\
&\geq \left( \frac{1}{N} \sum_{i=1}^N \|f_i\| \cdot \|g_i\| \right) \notag \\
&\geq \frac{1}{N} \sum_{i=1}^N \left| \langle f_i, g_i \rangle \right| \notag \\
&\geq \frac{n}{N}.
\end{align}

This chain of inequalities establishes a fundamental lower bound on the average Frobenius-norm error for $1$-erasure in terms of the dimension $n$ of the Hilbert space and the number of frame elements $N$. It emphasizes that even in the best-case scenario, the reconstruction error cannot be arbitrarily small, and is inherently limited by the geometry of the frame and its dual. \\
The following proposition establishes that when the frame $F$ is a uniform normalized tight frame (UNTF), the canonical dual frame is uniquely optimal under the Frobenius norm-based error measure. Specifically, for $p>2,$ the canonical dual minimizes the average error for $1-$erasure and is the unique minimizer. This uniqueness further extends to the case of 
$m-$erasures, showing that the canonical dual remains the optimal choice across all erasure levels.

  \begin{prop}
Let \( F = \{f_i\}_{i=1}^N \) be a UNTF for a finite-dimensional Hilbert space \( \mathcal{H}_n \). For \( p > 2 \), the canonical dual of \( F \) is the unique  $1$-erasure optimal dual frame under the Frobenius norm. Consequently, it is also the unique optimal dual frame for $m$-erasures, for all \( m \geq 1 \).
\end{prop}

          \begin{proof}
Let \( F = \{f_i\}_{i=1}^N \) be a uniform normalized tight frame for \( \mathcal{H}_n \). Then,  \( \|f_i\|^2 = \frac{n}{N} \) for all \( i \), and the canonical dual of \( F \) is itself, i.e., \( S_F^{-1}F = F \). Then the average Frobenius norm error for the frame $F$ and its canonical dual is
$$\mathrm{AE}_{\mathfrak{F}}^{(1),p}(F, F) = \left\{ \frac{1}{N} \sum_{i=1}^N \|f_i\|^{2p} \right\}^{1/p} = \left\{ \frac{1}{N} \cdot N \left( \frac{n}{N} \right)^p \right\}^{1/p} = \frac{n}{N}.$$

Now suppose \( G = \{g_i\}_{i=1}^N = \{f_i + u_i\}_{i=1}^N \) be a dual frame of \( F \) and assume it is also $1-$erasure optimal under the Frobenius norm for \( p > 2 \). We then have \( \mathrm{AE}_{\mathfrak{F}}^{(1),p}(F, G) = \frac{n}{N} \), and by the equation \eqref{eqn2point5}, we  have
\begin{align}\label{eqn3.1}
\left\{ \frac{1}{N} \sum_{i=1}^N \left( \|f_i\| \cdot \|f_i + u_i\| \right)^p \right\}^{1/p} = \frac{n}{N}.
\end{align}

Substituting \( \|f_i\| = \sqrt{n/N} \) in \eqref{eqn3.1}, we get:
\begin{align*}
\dfrac{1}{N} \sum_{i=1}^N \left( \frac{n}{N} \right)^{p/2} \|f_i + u_i\|^p = \left( \frac{n}{N} \right)^p.
\end{align*}

This leads to $\left( \dfrac{n}{N} \right)^{p/2} = \dfrac{1}{N} \sum\limits_{i=1}^N \|f_i + u_i\|^p.$ Now by using Theorem\ref{lemma2point4}, we have
\begin{align*}
\left( \frac{n}{N} \right)^{p/2} 
&= \frac{1}{N} \sum_{i=1}^N \left( \|f_i\|^2 + \|u_i\|^2 + 2 \mathrm{Re} \langle f_i, u_i \rangle \right)^{p/2} \\
&\geq \left( \frac{1}{N} \sum_{i=1}^N \left( \|f_i\|^2 + \|u_i\|^2 + 2 \mathrm{Re} \langle f_i, u_i \rangle \right) \right)^{p/2}.
\end{align*}

Using the fact that \( \sum_{i=1}^N \langle f, u_i \rangle f_i = 0 \) for all \( f \in \mathcal{H}_n \), we can get \( \sum_{i=1}^N \langle f_i, u_i \rangle = 0 .\)  Substituting \( \|f_i\|^2 = \frac{n}{N} \) and using the  condition  $\sum_{i=1}^N \langle f_i, u_i \rangle = 0 $, we obtain:
\[
\left( \frac{n}{N} \right)^{p/2} \geq \left( \frac{n}{N} + \frac{1}{N} \sum_{i=1}^N \|u_i\|^2 \right)^{p/2}.
\]

Since \( p > 2 \), the function \( x \mapsto x^{p/2} \) is strictly increasing, hence this inequality implies:
\[
\frac{n}{N} \geq \frac{n}{N} + \frac{1}{N} \sum_{i=1}^N \|u_i\|^2,
\]
which forces \( \|u_i\| = 0 \) for all \( i \). Therefore, \( G = F \), i.e., the canonical dual is the unique 1-erasure optimal dual under Frobenius norm when \( p > 2 \). The result similarly extends to \( m \)-erasures due to the convexity of the same function applied to higher-order sums.
\end{proof}

The following theorem provides a sufficient condition under which the canonical dual frame of $F$ attains one-erasure optimality with respect to the Frobenius norm. Let us define
$L := \max \left\{ \|f_i\| \cdot \|S_F^{-1} f_i\| : 1 \leq i \leq N \right\},$
and let $\eta_1 := \left\{ i : \|f_i\| \cdot \|S_F^{-1} f_i\| = L \right\}, \; \eta_2 := \{1,2,\ldots,N\} \setminus \eta_1.$ We also define the subspaces $H_j := \mathrm{span} \{ f_i : i \in \eta_j \}, \quad \text{for } j = 1,2.$

\begin{thm}\label{thm3point2}
Consider a frame \( F = \{f_i\}_{i=1}^N \) for a finite-dimensional Hilbert space \( \mathcal{H}_n \). Assume that the following conditions are satisfied:
\begin{itemize}
    \item [{\em (i)}]\( H_1 \cap H_2 = \{0\} \),
    \item [{\em (ii)}]\( \|f_i\| = c \) for all \( i \in \eta_1 \), for some constant \( c > 0 \),
    \item [{\em (iii)}]The set \( \{f_i : i \in \eta_2\} \) is linearly independent.
\end{itemize}
Then the canonical dual frame of \( F \) is a $1-$erasure optimal dual under the Frobenius norm for all \( p > 2 \).
\end{thm}
To prove the theorem, we require the following auxiliary lemma.

\begin{lem}\label{lemma4point2}
Let \( L > 0 \) be a positive real number and let \( \{L_i\}_{i=1}^s \) be a sequence of positive real numbers such that \( L_i \geq L \) for all \( 1 \leq i \leq s \). If
\[
\sum_{i=1}^s \left(L_i^2 + \alpha_i \right)^{p/2} \leq sL^p \quad \text{and} \quad \sum_{i=1}^s \alpha_i = 0,
\]
then for all \( i \), we have \( \alpha_i = 0 \) and \( L_i = L \), provided \( p \geq 2 \).
\end{lem}

\begin{proof}
We begin by observing that the convexity of the function \( x \mapsto x^{p/2} \) for \( p \geq 2 \) implies:
\[
\sum_{i=1}^s \left( L_i^2 + \alpha_i \right)^{p/2} \geq s \left( \frac{1}{s} \sum_{i=1}^s (L_i^2 + \alpha_i) \right)^{p/2} = s \left( \frac{1}{s} \sum_{i=1}^s L_i^2 \right)^{p/2}.
\]
Since \( L_i \geq L \), we further have
\[
\left( \frac{1}{s} \sum_{i=1}^s L_i^2 \right)^{p/2} \geq L^p,
\]
which together gives
\[
\sum_{i=1}^s \left( L_i^2 + \alpha_i \right)^{p/2} \geq sL^p.
\]
But the assumption provides the opposite inequality:
\[
\sum_{i=1}^s \left( L_i^2 + \alpha_i \right)^{p/2} \leq sL^p.
\]
Therefore, all inequalities must be equalities. In particular, equality in Jensen's inequality implies that each term \( L_i^2 + \alpha_i \) is equal to the mean:
\[
L_i^2 + \alpha_i = c \quad \text{for some constant } c \text{ and for all } i.
\]
So,
\[
sL^p = \sum_{i=1}^s \left( L_i^2 + \alpha_i \right)^{p/2} = \sum_{i=1}^s c^{p/2} = sc^{p/2}.
\]
This gives \( c^{p/2} = L^p \Rightarrow c = L^2 \), and hence \( L_i^2 + \alpha_i = L^2 \), which implies \( \alpha_i = L^2 - L_i^2 \leq 0 \) for all \( i \), since \( L_i \geq L \). Now, using the fact that \( \sum_{i=1}^s \alpha_i = 0 \), and each \( \alpha_i \leq 0 \), it follows that \( \alpha_i = 0 \) for all \( i \), which further implies \( L_i = L \) for all \( i \).
\end{proof}

\begin{proof}[\textbf{Proof of Theorem \ref{thm3point2}}]
Let \( G = \{ g_i \}_{i=1}^N \) be a \(1-\)erasure optimal dual of \( F \) with respect to the Frobenius norm.  $G$ can be  expressed as $g_i = S_F^{-1} f_i + u_i, \quad 1 \leq i \leq N,$ where the sequence \( \{u_i\}_{i=1}^N \) satisfies
\[
\sum_{i=1}^N \langle f, u_i \rangle f_i = 0 \quad \text{for all } f \in \mathcal{H}_n.
\]
\noindent This condition can be decomposed as
\[
\sum_{i \in \eta_1} \langle f, u_i \rangle f_i + \sum_{i \in \eta_2} \langle f, u_i \rangle f_i = 0.
\]
Since \( H_1 \cap H_2 = \{0\} \), this decomposition is unique and hence both sums must vanish individually, i.e.
\[
\sum_{i \in \eta_1} \langle f, u_i \rangle f_i = 0 \quad \text{and} \quad \sum_{i \in \eta_2} \langle f, u_i \rangle f_i = 0, \quad \text{for all } f \in \mathcal{H}_n.
\]

Now, the set \( \{f_i : i \in \eta_2\} \) is linearly independent, so the second condition implies \( \langle f, u_i \rangle = 0 \) for all \( i \in \eta_2 \) and all \( f \in \mathcal{H}_n \). This gives \( u_i = 0 \) for all \( i \in \eta_2 \). Let  \( U_1 = \{ u_i \}_{i \in \eta_1} \) and \( F_1 = \{ f_i \}_{i \in \eta_1} \). Then the condition $\sum_{i \in \eta_1} \langle f, u_i \rangle f_i = 0 \quad \text{for all } f \in \mathcal{H}_n$ implies that \( T_{F_1}^* T_{U_1} = 0 \), where \( T_{F_1} \) and \( T_{U_1} \) are the analysis operators corresponding to \( F_1 \) and \( U_1 \), respectively. Thus,
\[
\sum_{i \in \eta_1} \langle S_F^{-1} f_i, u_i \rangle = \mathrm{tr} \left( T_{U_1} T_{S_F^{-1} F_1}^* \right)
= \mathrm{tr} \left( S_F^{-1} T_{F_1}^* T_{U_1} \right) = 0.
\]

Since \( G \) is $1-$erasure optimal, we must have:
\begin{align}\label{eqnspec3point8}
 \left\{ \frac{1}{N} \sum_{i \in \eta_1} (\|f_i\| \cdot \|g_i\|)^p + \frac{1}{N}\sum_{i \in \eta_2} (\|f_i\| \cdot \|g_i\|)^p \right\}^{1/p} &=\mathrm{AE}_{\mathfrak{F}}^{(1),p}(F, G)  \nonumber\\
&\leq \mathrm{AE}_{\mathfrak{F}}^{(1),p}(F, S_F^{-1}F) \nonumber\\&
\left\{ \frac{1}{N} \sum_{i \in \eta_1} (\|f_i\| \cdot \|S_F^{-1} f_i\|)^p + \frac{1}{N}\sum_{i \in \eta_2} (\|f_i\| \cdot \|S_F^{-1} f_i\|)^p \right\}^{1/p}.
\end{align}

Using the definition of \( L \), we have $\sum\limits_{i \in \eta_1} (\|f_i\| \cdot \|g_i\|)^p \leq |\eta_1| L^p.$ Substituting \( g_i = S_F^{-1} f_i + u_i \), we have:
\begin{align*}
\sum_{i \in \eta_1} (\|f_i\| \cdot \|g_i\|)^p
&= \sum_{i \in \eta_1} \left( \|f_i\| \cdot \| S_F^{-1} f_i + u_i \| \right)^p \\
&= \sum_{i \in \eta_1} \left( \|f_i\|^2 \| S_F^{-1} f_i \|^2 + \|f_i\|^2 \|u_i\|^2 + 2 \|f_i\|^2 \mathrm{Re} \langle S_F^{-1} f_i, u_i \rangle \right)^{p/2} \\
&= \sum_{i \in \eta_1} \left( L^2 + \|f_i\|^2 \|u_i\|^2 + 2 \|f_i\|^2 \mathrm{Re} \langle S_F^{-1} f_i, u_i \rangle \right)^{p/2}.
\end{align*}

Therefore from \eqref{eqnspec3point8}, we have
\[
\sum_{i \in \eta_1} \left( L^2 + \|f_i\|^2 \|u_i\|^2 + 2 \|f_i\|^2 \mathrm{Re} \langle S_F^{-1} f_i, u_i \rangle \right)^{p/2} \leq |\eta_1| L^p.
\]

By Lemma~\ref{lemma4point2}, the condition $\sum_{i \in \eta_1} \langle S_F^{-1} f_i, u_i \rangle = 0$ implies that \( \|u_i\| = 0 \) for all \( i \in \eta_1 \). Consequently, \( u_i = 0 \) for all \( i \in \eta_1 \). Since \( u_i = 0 \) on \( \eta_1 \) and the remaining components are already zero, we obtain \( u_i = 0 \) for every \( 1 \leq i \leq N \). Thus, \( G = S_F^{-1} F \), establishing that the canonical dual is the unique \(1-\)erasure optimal dual with respect to the Frobenius norm.

\end{proof}

Now, consider the case when erasures occur at multiple indices. For \( 1 < m \leq N \), suppose errors occur at positions \( i_1, i_2, \dots, i_m \). Then, for any dual frame \( G' = \{g'_i\}_{i=1}^N \) of \( F \), the Frobenius norm of the associated error operator satisfies:

\begin{align} \label{eqn4point8}
\left\|E_{\Lambda, F, G'}\right\|_{\mathfrak{F}} 
&= \left\|T_{G'}^* D T_F\right\|_{\mathfrak{F}} \nonumber \\
&= \sqrt{\mathrm{tr}\left(D T_F T_F^* D T_{G'} T_{G'}^*\right)} \nonumber \\
&= \sqrt{ \sum_{r=1}^m \langle g'_{i_1}, g'_{i_r} \rangle \langle f_{i_r}, f_{i_1} \rangle 
      + \sum_{r=1}^m \langle g'_{i_2}, g'_{i_r} \rangle \langle f_{i_r}, f_{i_2} \rangle 
      + \cdots 
      + \sum_{r=1}^m \langle g'_{i_m}, g'_{i_r} \rangle \langle f_{i_r}, f_{i_m} \rangle } \nonumber \\
&= \sqrt{ \sum_{r=1}^m \|g'_{i_r}\|^2 \|f_{i_r}\|^2 
      + \sum_{\substack{j,k=1 \\ j \neq k}}^m \langle g'_{i_j}, g'_{i_k} \rangle \langle f_{i_k}, f_{i_j} \rangle } \nonumber \\
&= \sqrt{ \sum_{r=1}^m \|g'_{i_r}\|^2 \|f_{i_r}\|^2 
      + 2\, \mathrm{Re} \left( \sum_{\substack{j,k=1 \\ j > k}}^m \langle g'_{i_j}, g'_{i_k} \rangle \langle f_{i_k}, f_{i_j} \rangle \right) }.
\end{align}

  Now, we present an explicit expression for the Frobenius-norm-based average error for any $m-$ erasures when the underlying frame is a uniform tight frame with constant modulus of inner products. The result also establishes the uniqueness of the canonical dual as the optimal dual frame under such conditions. 
   
     \begin{thm}\label{thm3point4}
Let \( F = \{f_i\}_{i=1}^N \) be a uniform tight frame for the Hilbert space \( \mathcal{H}_n \), such that \( |\langle f_i, f_j \rangle| \) is constant for all \( i \neq j \). Then, for any \( 1 \leq m \leq N \) and \( p > 0 \), the average Frobenius-norm error over all \( m \)-erasure patterns satisfies $\mathrm{AE}_{\mathfrak{F}}^{(m),p}(F) = \sqrt{m\left(\frac{n}{N}\right)^2 + \frac{m(m-1)(nN - n^2)}{N^2(N-1)}}. $
Moreover, the canonical dual frame \( S_F^{-1} F \) is the unique optimal dual of \( F \) under the Frobenius norm for any \( m-\)erasure. 
\end{thm}

   \begin{proof}
We first prove the result for the case \( m = 1 \). Let \( F \) be a uniform tight frame with frame bound \( A \). Then each frame vector satisfies \( \|f_i\| = \sqrt{\frac{An}{N}} \) for all \( 1 \leq i \leq N \). The canonical dual frame is \( S_F^{-1}F = \left\{ \frac{1}{A}f_i \right\}_{i=1}^N \). Then the average Frobenius-norm error becomes
\[
\mathrm{AE}_{\mathfrak{F}}^{(1),p}(F, S_F^{-1}F) = \left\{ \frac{1}{N} \sum_{i=1}^N \left\|f_i\right\|^p \left\|\frac{1}{A} f_i \right\|^p \right\}^{1/p} = \left\{ \frac{1}{N A^p} \sum_{i=1}^N \|f_i\|^{2p} \right\}^{1/p}.
\]
Since \( \|f_i\| = \sqrt{\frac{An}{N}} \), we obtain
\[
\mathcal{F}_1^p(F, S_F^{-1}F) = \left\{ \frac{1}{N A^p} \cdot N \left( \frac{An}{N} \right)^p \right\}^{1/p} = \frac{n}{N}.
\]
Hence by \eqref{eqn4.6}, the canonical dual $S_F^{-1}F$ is a $1-$erasure optimal dual of \( F \) under the Frobenius norm.

Now consider the case \( m > 1 \). Let \( G \in \zeta_{\mathfrak{F}}^{(m),p}(F) \). Then \( G \in \zeta_{\mathfrak{F}}^{(1),p}(F) \), and thus,
\[
\mathrm{AE}_{\mathfrak{F}}^{(1),p}(F, G)=\left\{ \frac{1}{N} \sum_{i=1}^N \|f_i\|^p \|g_i\|^p \right\}^{1/p} = \frac{n}{N}.
\]
Now,
\[
N \left(\frac{n}{N}\right)^p = \sum_{i=1}^N \|f_i\|^p \|g_i\|^p \geq N \left( \frac{\sum_{i=1}^N \|f_i\| \|g_i\|}{N} \right)^p \geq N \left( \frac{\sum_{i=1}^N |\langle f_i, g_i \rangle|}{N} \right)^p \geq N \left( \frac{n}{N} \right)^p.
\]
Thus, equality holds throughout, which leads to $\|f_i\| \|g_i\| = |\langle f_i, g_i \rangle| = \frac{n}{N}, \quad \forall i.$ For a frame $F=\{f_i\}_{i=1}^N$ and a dual $G=\{g_i\}_{i=1}^N,$ we have

\[
\sum_{i,j=1}^N \langle g_i, g_j \rangle \langle f_j, f_i \rangle = \sum_{i=1}^N \left\langle g_i, \sum_{j=1}^N \langle f_i, f_j \rangle g_j \right\rangle = \sum_{i=1}^N \langle g_i, f_i \rangle = n.
\]
Also, since \( \|f_i\| \|g_i\| = \frac{n}{N} \), we derive:
\begin{align} \label{equation3point8}
2\,\text{Re} \left( \sum_{\substack{i,j=1 \\ i > j}}^N \langle g_i, g_j \rangle \langle f_j, f_i \rangle \right) = \sum_{i \neq j} \langle g_i, g_j \rangle \langle f_j, f_i \rangle = n - \frac{n^2}{N}.
\end{align}

Now, from equation~\eqref{eqn4point8}, 
\begin{align*}
\mathrm{AE}_{\mathfrak{F}}^{(m),p}(F, G)
&= \left\{ \frac{1}{\binom{N}{m}} \sum_{D \in \mathcal{D}^{(m)}} \left( \sum_{r=1}^m \|g_{i_r}\|^2 \|f_{i_r}\|^2 + \sum_{\substack{j,k=1 \\ j \neq k}}^m \langle g_{i_j}, g_{i_k} \rangle \langle f_{i_k}, f_{i_j} \rangle \right)^{p/2} \right\}^{1/p} \\
&= \left\{ \frac{1}{\binom{N}{m}} \sum_{D \in \mathcal{D}^{(m)}} \left( m \left( \frac{n}{N} \right)^2 + \sum_{\substack{j,k=1 \\ j \neq k}}^m \langle g_{i_j}, g_{i_k} \rangle \langle f_{i_k}, f_{i_j} \rangle \right)^{p/2} \right\}^{1/p} \\
&\geq \left( m \left( \frac{n}{N} \right)^2 + \frac{m(m-1)}{N(N-1)} \sum_{i \neq j} \langle g_i, g_j \rangle \langle f_j, f_i \rangle \right)^{1/2} \\
&= \sqrt{ m\left( \frac{n}{N} \right)^2 + \frac{m(m-1)(nN - n^2)}{N^2(N-1)} }.
\end{align*}

Thus, $\mathrm{AE}_{\mathfrak{F}}^{(1),p}(F) \geq \sqrt{ m\left( \dfrac{n}{N} \right)^2 + \dfrac{m(m-1)(nN - n^2)}{N^2(N-1)} }.$ Also,
\begin{align*}
\mathrm{AE}_{\mathfrak{F}}^{(1),p}(F, S_{F}^{-1}F)
&= \left\{ \frac{1}{\binom{N}{m}} \sum_{D \in \mathcal{D}^{(m)}} \left( \sum_{r=1}^m \frac{1}{A^2} \|f_{i_r}\|^4 + \frac{1}{A^2} \sum_{\substack{j,k=1 \\ j \neq k}}^m |\langle f_{i_j}, f_{i_k} \rangle|^2 \right)^{p/2} \right\}^{1/p} \\
&= \sqrt{ m\left( \dfrac{n}{N} \right)^2 + \dfrac{m(m-1)(nN - n^2)}{N^2(N-1)} }.
\end{align*}

Hence, $ \mathrm{AE}_{\mathfrak{F}}^{(m),p}(F) = \sqrt{ m\left( \dfrac{n}{N} \right)^2 + \dfrac{m(m-1)(nN - n^2)}{N^2(N-1)} }.$  It remains to prove the uniqueness. Suppose there exists another optimal dual \( \tilde{G} = \left\{ \frac{1}{A} f_i + u_i \right\}_{i=1}^N \) satisfying \(  \mathrm{AE}_{\mathfrak{F}}^{(1),p}(F, \tilde{G}) = \mathrm{AE}_{\mathfrak{F}}^{(1),p}(F) = \frac{n}{N} \). Therefore, $\frac{n}{N}=\left\{\frac{1}{N} \sum \limits_{i=1}^N\|f_i\|^p\,\| \tilde{g}_i \| ^p\right\}^{\frac{1}{p}} \geq \dfrac{\sum \limits_{i=1}^N\|f_i\|^p\,\| \tilde{g}_i \|}{N} \geq \frac{n}{N}. $ This leads to $\left\{\frac{1}{N} \sum \limits_{i=1}^N\|f_i\|^p\,\| \tilde{g}_i \| ^p\right\}^{\frac{1}{p}} = \dfrac{\sum \limits_{i=1}^N\|f_i\|\,\| g_i \|}{N} = \frac{n}{N}.$ This is possible if and only if  each $\|f_i\|\,\| \tilde{g}_i \|$ are equal and hence, $\|f_i\|\,\| \tilde{g}_i \| = \|f_i\|\left\| \frac{1}{A}f_i + u_i \right\|= \frac{n}{N},$ for all $1 \leq i \leq N.$ This leads to, $\left\| \frac{1}{A}f_i + u_i \right\| = \sqrt{\frac{n}{AN}},\;1 \leq i \leq N.$ Explicitly, $\frac{1}{A^2}\|f_i\|^2 + \frac{2}{A} Re\langle f_i, u_i \rangle + \|u_i\|^2 = \frac{n}{AN},\;1 \leq i \leq N.$ Thus, $\frac{1}{A^2}\sum\limits_{1 \leq i \leq N} \|f_i\|^2 + \frac{2}{A} \sum\limits_{1 \leq i \leq N} Re\langle f_i, u_i \rangle + \sum\limits_{1 \leq i \leq N} \|u_i\|^2  = \frac{n}{A}.$ Consequently,
 \begin{align}\label{equation9}
 	\frac{2}{A} \sum\limits_{1 \leq i \leq N} Re\langle f_i, u_i \rangle + \sum\limits_{1 \leq i \leq N} \|u_i\|^2 = 0.
 \end{align}
As $\tilde{G}$ is a dual of $F,$  $\sum\limits_{1\leq i \leq N} \langle f,f_i \rangle u_i =0,$ for all $f \in \mathcal{H}_n.$ This implies, $T_{U}^* T_F = 0,$ where $U= \{u_i\}_{i=1}^N.$ Therefore, $0 = tr(T_{U}^* T_F) = tr(T_F T_{U}^*) = \sum\limits_{1\leq i \leq N} \langle f_i,u_i \rangle. $ So, $Re\left( \sum\limits_{1\leq i \leq N} \langle f_i,u_i \rangle \right) =0.$ Hence, by \eqref{equation9}, $\sum\limits_{1\leq i \leq N} \|u_i\|^2 = 0,$ which implies, $u_i = 0,$ for all $1 \leq i \leq N.$ Thus, $\tilde{G} = S_{F}^{-1}F $.

\end{proof}

The next corollary follows directly from the preceding theorem. It shows that for an equiangular tight frame, the canonical dual serves as the unique Frobenius-norm optimal dual frame, regardless of the number of erasures.

\begin{cor}
Suppose $F = \{f_i\}_{i=1}^N$ is a uniform tight frame for $\mathcal{H}_n$ with the property that $|\langle f_i, f_j \rangle|$ is a constant for all \( i \neq j \). Then every element in $ \zeta_{\mathfrak{F}}^{(1),p}(F)$ is a $1-$uniform dual of $F$.
\end{cor}

\begin{proof}
From the proof of Theorem~\ref{thm3point4}, if $G$ is a $1$-erasure optimal dual frame of $F$ under the Frobenius norm, then $\|f_i\| \|g_i\| = |\langle f_i, g_i \rangle| = \frac{n}{N}, \quad \forall\, i.$ Let $\langle f_j, g_j \rangle = a_j + i b_j$ for $1 \leq j \leq N$. Then $\sum_{j=1}^N a_j = n, \quad  \sum_{j=1}^N b_j = 0, \quad \text{and} \quad  \sqrt{a_j^2 + b_j^2} = \frac{n}{N} \quad \text{for all } 1 \leq j \leq N.$ It follows that $\sum_{j=1}^N a_j = \sum_{j=1}^N \sqrt{a_j^2 + b_j^2},$ which implies $b_j = 0$ for all $j$. Suppose, for contradiction, that $a_j < \frac{n}{N}$ for some $j$. Then there must exist an index $\ell$ such that $a_\ell > \frac{n}{N}$, which gives $|\langle f_\ell, g_\ell \rangle| = \sqrt{a_\ell^2 + b_\ell^2} > \frac{n}{N},$ a contradiction. Hence, $a_j = \frac{n}{N}$ and $b_j = 0$ for all $1 \leq j \leq N$. Consequently, $\langle f_j, g_j \rangle = \frac{n}{N}, \quad \forall\, 1 \leq j \leq N,$ showing that $G$ is a one-uniform dual frame of $F$.
\end{proof}

\section{Numerical Radius as Error Measure}

In this section, we consider the \emph{numerical radius} as the measure of error for the error operator. Let \( F = \{f_i\}_{i=1}^N \) be a frame for a finite-dimensional Hilbert space \( \mathcal{H}_n \), and let \( G = \{g_i\}_{i=1}^N \) be a dual frame of \( F \). The numerical radius of a bounded operator \( T \) is given by $\omega(T) := \sup \left\{ \left| \langle Tf, f \rangle \right| : \|f\| = 1 \right\}.$ Given the error operator \( E_{\Lambda, F, G} := T_G^* D T_F \) associated with an erasure location \( \Lambda \subset \{1, \dots, N\} \) of size \( m \), we define the average numerical-radius-based error over all \( m-\)erasure patterns as
\[
\mathrm{AE}_{\mathcal{N}}^{(m),p}(F, G) := \left\{ \frac{1}{\binom{N}{m}} \sum_{D \in \mathcal{D}^{(m)}} \left[ \omega\left( T_G^* D T_F \right) \right]^p \right\}^{1/p},
\]
where \( \mathcal{D}^{(m)} \) denotes the collection of all \( N \times N \) diagonal matrices whose diagonal entries consist of exactly \( m \) ones and \( N - m \) zeros. The corresponding optimal error value over all duals of \( F \) is defined as
\[
\mathrm{AE}_{\mathcal{N}}^{(1),p}(F) := \inf \left\{ \mathrm{AE}_{\mathcal{N}}^{(1),p}(F, G) : G \text{ is a dual of } F \right\},
\]
and the set of duals achieving this infimum is denoted by
\[
\zeta_{\mathcal{N}}^{(1),p}(F) := \left\{ G : \mathrm{AE}_{\mathcal{N}}^{(1),p}(F, G) = \mathrm{AE}_{\mathcal{N}}^{(1),p}(F) \right\}.
\]
Every element in \( \zeta_{\mathcal{N}}^{(1),p}(F) \) is called a \textit{\( 1-\)erasure numerically optimal dual of \( F \)}.

\noindent For \( m = 1 \), if the error occurs at the \( i^{\text{th}} \) position, the error operator becomes $E_{\Lambda, F, G} f = \langle f, f_i \rangle g_i.$ By Lemma 2.1 in \cite{chien}, the numerical radius of this operator is $\omega\left( E_{\Lambda, F, G} \right) = \frac{\left| \langle f_i, g_i \rangle \right| + \|f_i\| \cdot \|g_i\|}{2}.$
Hence, the average numerical error for \( m = 1 \) becomes
\begin{equation}\label{eqn5point14}
   \mathrm{AE}_{\mathcal{N}}^{(1),p}(F, G) = \left\{ \frac{1}{N} \sum_{i=1}^N \left( \frac{|\langle f_i, g_i \rangle| + \|f_i\| \cdot \|g_i\|}{2} \right)^p \right\}^{1/p}. 
\end{equation}

   \begin{prop}
    Let \( F = \{f_i\}_{i=1}^N \) be a uniform Parseval frame for a Hilbert space \( \mathcal{H}_n \). Then the value of \( \mathrm{AE}_{\mathcal{N}}^{(1),p}(F) \) is \( \dfrac{n}{N} \), and the canonical dual is the unique 1-erasure numerically optimal dual of \( F \).
\end{prop}

\begin{proof}
Since \( F \) is a uniform Parseval frame, we have \( \|f_i\| = \sqrt{\frac{n}{N}} \) for all \( i = 1, \dots, N \). For any dual frame \( G = \{g_i\}_{i=1}^N \) of \( F \), by \eqref{eqn5point14}, we have

\[
\mathrm{AE}_{\mathcal{N}}^{(1),p}(F, G) = \left\{ \frac{1}{N} \sum_{i=1}^N \left( \frac{|\langle g_i, f_i \rangle| + \|g_i\| \cdot \|f_i\|}{2} \right)^p \right\}^{1/p}.
\]
Using the inequalities $\dfrac{|\langle g_i, f_i \rangle| + \|g_i\| \cdot \|f_i\|}{2} \geq |\langle g_i, f_i \rangle| \quad \text{and} \quad \sum_{i=1}^N |\langle g_i, f_i \rangle| \geq \left| \sum_{i=1}^N \langle g_i, f_i \rangle \right| = n,$ we obtain $\mathrm{AE}_{\mathcal{N}}^{(1),p}(F, G) \geq \frac{n}{N}.$ If we take $G$ as $S_{F}^{-1}F= F$, then $\mathrm{AE}_{\mathcal{N}}^{(1),p}(F, F) = \left\{ \frac{1}{N} \sum_{i=1}^N \left( \|f_i\|^2 \right)^p \right\}^{1/p} = \frac{n}{N}. $ Hence, $\mathrm{AE}_{\mathcal{N}}^{(1),p}(F) = \frac{n}{N}.$

Now, let \( \tilde{G} = \{\tilde{g}_i\} = \{S_F^{-1} f_i + u_i\} \in \zeta_{\mathcal{N}}^{(1),p}(F) \), where \( \{u_i\} \) satisfies \( \sum_{i=1}^N \langle f, f_i \rangle u_i = 0 \) for all \( f \in \mathcal{H}_n \). Then, $\mathrm{AE}_{\mathcal{N}}^{(1),p}(F, \tilde{G}) = \frac{n}{N}.$ This gives, 
        $$\frac{n}{N}= \left\{\frac{1}{N} \sum\limits_{i=1}^N\left( \frac{|\langle g_i,f_i\rangle| + \|g_i\|\|f_i\|}{2}\right)^p\right\}^{\frac{1}{p}} \geq \left\{\frac{1}{N} \sum\limits_{i=1}^N|\langle g_i,f_i\rangle|^p\right\}^{\frac{1}{p}} \geq \dfrac{\sum\limits_{i=1}^N|\langle g_i,f_i\rangle|}{N} \geq \dfrac{n}{N}.  $$
       Therefore,  $\dfrac{|\langle g_i,f_i\rangle| + \|g_i\|\|f_i\|}{2}$ is a constant, for all $i.$   This leads to,  $\dfrac{|\langle g_i,f_i\rangle| + \|g_i\|\|f_i\|}{2} = \frac{n}{N},\,1\leq i \leq N.$ Now, $n= \sum\limits_{i=1}^N \langle g_i,f_i \rangle \leq \sum\limits_{i=1}^N \left|\langle g_i,f_i \rangle \right| \leq \sum\limits_{i=1}^N \dfrac{|\langle g_i,f_i\rangle| + \|g_i\|\|f_i\|}{2} = n. $ Therefore, $\sum\limits_{i=1}^N \left|\langle g_i,f_i \rangle \right| = \sum\limits_{i=1}^N \dfrac{|\langle g_i,f_i\rangle| + \|g_i\|\|f_i\|}{2} = n.$ Using the fact $\left|\langle g_i,f_i \rangle \right| \leq \|f_i\|\,\|g_i\|,$ we have 
\begin{align}\label{equation10}
	\left|\langle g_i,f_i \rangle \right| = \|f_i\|\,\|g_i\| = \frac{n}{N} ,\;\forall 1\leq i \leq N.
\end{align}

Let $\langle g_j,f_j\rangle = a_j +ib_j,\,$ where $i = \sqrt{-1}.$ Then $\sum\limits_{j=1}^N a_j =n,\,\sum\limits_{j=1}^N b_j=0.$ Therefore by \eqref{equation10}, we have $\sqrt{a^2_{j}+ b^2_{j} } = \frac{n}{N}, \,\forall j, $ which shows that $b_j =0\,\forall j.$ Thus, $\sum\limits_{j=1}^N a_j = \sum\limits_{j=1}^N \left|a_j \right| $ and hence, $a_j \geq 0 \forall j.$ Therefore, $\langle g_j,f_j\rangle= \frac{n}{N}, \forall j.$ 

Now, $\sum\limits_{i=1}^N \langle f_i,u_i \rangle = tr(T_U T_{F}^*)= tr(T_{F}^* T_U) = 0.$  Also, $\frac{n}{N} = |\langle g_i,f_i\rangle| = \left| \|  S_{F}^{-\frac{1}{2}}f_i \|^2 + \langle u_i,f_i\rangle    \right| = \left| \frac{n}{N} + \langle u_i,f_i\rangle   \right| ,\, 1 \leq i \leq N.$ This implies that,
 	\begin{align}\label{eqn5point14}
 		\sqrt{\left( \frac{n}{N} + Re \langle u_i,f_i\rangle  \right)^2 + \left( Im \langle u_i,f_i\rangle \right)^2} = \frac{n}{N}, \;1 \leq i \leq N.
 	\end{align}

 Now, we will show that $\langle u_i,f_i\rangle =0,\,1 \leq i \leq N.$ If for any $j,\, Re \langle u_j,f_j\rangle < 0,$ then there exists a $1\leq j'\leq N,$ such that  $ Re \langle u_j',f_j'\rangle > 0,\,$ as $\sum\limits_{i=1}^N Re\left(\langle f_i,u_i \rangle \right) =0.$  This leads to, $\sqrt{\left( \frac{n}{N} + Re \langle u_j',f_j'\rangle  \right)^2 + \left( Im \langle u_j',f_j'\rangle \right)^2} > \frac{n}{N},$ which is a contradiction. Therefore, $ Re \langle u_i,f_i\rangle  = 0, \forall i.$ By \eqref{eqn5point14}, we have $ Im \langle u_i,f_i\rangle  = 0, \forall i$ and hence, $\langle u_i,f_i\rangle  = 0, \forall i.$
It is easy to see that $\left\{ S_{F}^{-\frac{1}{2}}f_i + S_{F}^{-\frac{1}{2}}u_i\right\}_{i=1}^N$ is a dual of $\left\{ S_{F}^{-\frac{1}{2}}f_i \right\}_{i=1}^N,$ as $\sum\limits_{i=1}^N \langle f, S_{F}^{-\frac{1}{2}}u_i \rangle S_{F}^{-\frac{1}{2}}f_i = S_{F}^{-\frac{1}{2}}\left(\sum\limits_{i=1}^N\langle S_{F}^{-\frac{1}{2}}f,u_i\rangle f_i  \right) = 0,\,\forall f \in \mathcal{H}_n. $ Then, $0=\sum\limits_{i=1}^N \langle S_{F}^{-\frac{1}{2}}u_i, S_{F}^{-\frac{1}{2}}f_i \rangle= \sum\limits_{i=1}^N \langle S_{F}^{-1}u_i, f_i \rangle. $ Now,
 	
 \begin{align*}
 &\;\;\;\;\;\;\;\;	\|f_i\|\,\|g_i\|  = \frac{n}{N} = 	\|f_i\|\,\|S_{F}^{-1}f_i\|, \,\forall i\\&
 	\implies \| S_{F}^{-1}f_i + u_i \|^2 = \|S_{F}^{-1}f_i\|^2,\,\forall i \\&
 	\implies \|u_i\|^2 + 2 Re \langle S_{F}^{-1}u_i, f_i \rangle =0,\,\forall i \\&
 	\implies \sum\limits_{i=1}^N \|u_i\|^2 + 2 Re \left(\sum\limits_{i=1}^N \langle S_{F}^{-1}u_i, f_i \rangle \right) =0 \\&
 	\implies  \sum\limits_{i=1}^N \|u_i\|^2 =0 \\&
 	\implies u_i =0,\, \forall i.
 \end{align*}
 	
 	Hence,  $G = S_{F}^{-1}F. $
    
\end{proof}

    \begin{rem}
         For  a uniform Parseval frame $F,$ the canonical dual $S_{F}^{-1}F =F$ is a $1-$erasure Numerically optimal dual frame of $F.$
    \end{rem}

The following proposition establishes that, for a tight frame, membership of the canonical dual in \( \zeta_{\mathcal{N}}^{(1),p}(F) \) implies its membership in \( \zeta_{\mathfrak{F}}^{(1),p}(F) \).

 \begin{prop}\label{prop6point2}
	Let $F$ be a tight frame for $\mathcal{H}_n.$ If $ S_{F}^{-1}F\in \zeta_{\mathcal{N}}^{(1),p}(F),$ then  $ S_{F}^{-1}F\in \zeta_{\mathfrak{F}}^{(1),p}(F).$ 
	
\end{prop}

\begin{proof}
    Let \( F \) be a tight frame with frame bound \( A \). Then, for any dual \( G \) of \( F \), we have
\begin{align*}
    \mathrm{AE}_{\mathfrak{F}}^{(1),p}(F, S_{F}^{-1}F) &=\left\{\frac{1}{N} \sum \limits_{i=1}^N\|f_i\|^p\,\|  S_{F}^{-1}f_i \| ^p\right\}^{\frac{1}{p}}\\&
    = \left\{\frac{1}{N} \sum\limits_{i=1}^N\dfrac{1}{A^p} \|f_i\|^{2p}\right\}^{\frac{1}{p}} \\&
    =  \mathrm{AE}_{\mathcal{N}}^{(1),p}(F, S_{F}^{-1}F) \\&
    \leq\mathrm{AE}_{\mathcal{N}}^{(1),p}(F, G) \\&
    \leq \left\{\frac{1}{N} \sum\limits_{i=1}^N\left( \frac{|\langle g_i,f_i\rangle| + \|g_i\|\|f_i\|}{2}\right)^p\right\}^{\frac{1}{p}}\\&
    \leq \left\{\frac{1}{N} \sum \limits_{i=1}^N\|f_i\|^p\,\|  g_i \| ^p\right\}^{\frac{1}{p}}\\&
    = \mathrm{AE}_{\mathfrak{F}}^{(1),p}(F, G) .
\end{align*}
 Hence, $ S_{F}^{-1}F\in \zeta_{\mathfrak{F}}^{(1),p}(F).$    
  
\end{proof}

\section{Spectral Radius as Error Measurement}

\noindent In this section, we investigate the problem of identifying dual frames that minimize the average reconstruction error when a fixed number of erasures occur, using the \textit{spectral radius} as the error measurement. Let $F = \{f_i\}_{i=1}^N$ be a frame for a finite-dimensional Hilbert space $\mathcal{H}_n$, and let $G = \{g_i\}_{i=1}^N$ be a dual frame of $F$. For $p > 1$ and $m \geq 1$, we define the average spectral-radius error over all $m$-erasure patterns as:
\begin{align*}
\mathrm{AE}_{\mathfrak{R}}^{(m),p}(F, G) 
&:= \left\{ \frac{1}{\binom{N}{m}} \sum\limits_{D \in \mathcal{D}^{(m)}} \left[ \rho\left( T_G^* D T_F \right) \right]^p \right\}^{1/p},
\end{align*}
where $\rho(\cdot)$ denotes the spectral radius and $\mathcal{D}^{(m)}$ is the set of all diagonal matrices in $\mathbb{C}^{N \times N}$ with exactly $m$ ones and $N-m$ zeros on the diagonal.

The corresponding optimal spectral error among all duals is given by
\begin{align*}
\mathrm{AE}_{\mathfrak{R}}^{(m),p}(F) 
&:= \inf \left\{ \mathrm{AE}_{\mathfrak{R}}^{(m),p}(F, G) : G \text{ is a dual of } F \right\}.
\end{align*}

The collection of all dual frames attaining this infimum is defined as
\begin{align*}
\zeta_{\mathfrak{R}}^{(m),p}(F) 
&:= \left\{ G : \mathrm{AE}_{\mathfrak{R}}^{(m),p}(F, G) = \mathrm{AE}_{\mathfrak{R}}^{(m),p}(F) \right\}.
\end{align*}

Each element of $\zeta_{\mathfrak{R}}^{(m),p}(F)$ is called a \textit{$m$-erasure spectrally optimal dual frame of $F$ under the spectral radius}.

\noindent \noindent
For a $(N,n)$ dual pair $(F,G),$ if the error occurs in the $i^{th}$ position, then $\rho\left(T_{G}^*DT_{F}\right) = \left|\langle f_i, g_i \rangle \right| $ and and thus the average spectral-radius error reduces to
	\begin{equation}\label{equation2}
		\mathrm{AE}_{\mathfrak{R}}^{(1),p} (F,G) = \left\{\frac{1}{N} \sum \left|\langle f_i, g_i \rangle \right|^p\right\}^{\frac{1}{p}}
	\end{equation}


In the study of erasures, it is particularly important to understand the structure of dual frames that minimize the spectral radius of the reconstruction error. For uniform frames, the following proposition characterizes the duals that are optimal for one erasure under the spectral radius.

\begin{prop}\label{prop3point2}
Let \( \mathcal{H}_n \) be an \( n-\)dimensional Hilbert space and let \( F = \{f_i\}_{i=1}^N \) be a \( 1-\)uniform frame for \( \mathcal{H}_n \). A dual frame \( G = \{g_i\}_{i=1}^N \) is a \( 1-\)erasure spectrally optimal dual of \( F \) if and only if \( G \) is a corresponding \( 1-\)uniform dual of $F.$
\end{prop}

\begin{proof}
Let \( \tilde{G} = \{\tilde{g}_i\}_{i=1}^N \) be a dual frame of \( F \). Then by definition,
\begin{align}\label{eqn3.3}
\mathrm{AE}_{\mathfrak{R}}^{(1),p}(F, \tilde{G}) = \left\{ \frac{1}{N} \sum_{i=1}^N \left| \langle f_i, \tilde{g}_i \rangle \right|^p \right\}^{1/p} \geq \left( \frac{1}{N} \sum_{i=1}^N \left| \langle f_i, \tilde{g}_i \rangle \right| \right) \geq \frac{1}{N} \left| \sum_{i=1}^N \langle f_i, \tilde{g}_i \rangle \right| = \frac{n}{N}.
\end{align}
The the last equality follows from the duality condition \( \sum_{i=1}^N \langle f_i, \tilde{g}_i \rangle = n \). Now, if \( G \) is a 1-uniform dual, then \( \langle f_i, g_i \rangle = \frac{n}{N} \) for all \( i \), and so
\[
\mathrm{AE}_{\mathfrak{R}}^{(1),p}(F, G) 
= \left\{ \frac{1}{N} \sum_{i=1}^N \left( \frac{n}{N} \right)^p \right\}^{1/p} 
= \frac{n}{N}.
\]
Hence, \( G \) is 1-erasure spectrally optimal dual of $F$.

Conversely, suppose \( G'' = \{g_i''\}_{i=1}^N \) is a 1-erasure spectrally optimal dual, so that
\(
\mathrm{AE}_{\mathfrak{R}}^{(1),p}(F, G'') = \frac{n}{N}.
\)
From \eqref{eqn3.3}, this implies equality in the chain of inequalities, which occurs only if \( \left| \langle f_i, g_i'' \rangle \right| = \frac{n}{N} \) for all \( i \). Write \( \langle f_i, g_i'' \rangle = a_i + ib_i \), with \( a_i, b_i \in \mathbb{R} \). Then
\[
\sum_{i=1}^N a_i = n \quad \text{and} \quad \sum_{i=1}^N \sqrt{a_i^2 + b_i^2} = n.
\]
Since \( \sqrt{a_i^2 + b_i^2} \geq a_i \), equality implies \( b_i = 0 \) and \( a_i \geq 0 \), hence \( \langle f_i, g_i'' \rangle = \frac{n}{N} \) for all \( i \), so \( G'' \) is also a 1-uniform dual of \( F \).
\end{proof}

As a direct consequence of the above Proposition the following corollary confirms that the canonical dual of a uniform tight frame is spectrally optimal for $1-$erasure.

\begin{cor}
Let \( F = \{f_i\}_{i=1}^N \) be a uniform tight frame for a Hilbert space \( \mathcal{H}_n \). Then the canonical dual frame \( S_F^{-1}F \) is a 1-erasure spectrally optimal dual of \( F \).
\end{cor}

    \begin{proof}
Let \( F = \{f_i\}_{i=1}^N \) be a uniform tight frame for the Hilbert space \( \mathcal{H}_n \) with frame bound \( A \). Then, by the definition of a uniform tight frame, each frame element satisfies $\|f_i\| = \sqrt{\frac{An}{N}}, \quad \text{for all } i = 1, \dots, N.$ It follows that $\langle f_i, S_{F}^{-1}f_i \rangle = \left\langle f_i, \frac{1}{A} f_i \right\rangle = \frac{1}{A} \|f_i\|^2 = \frac{n}{N}.$ Hence, by Proposition\ref{prop3point2}, the canonical dual is a $1-$erasure spectrally optimal dual of \( F \).
\end{proof}

\begin{rem}
		If $F$ has the uniform redundancy distribution, then there exist a dual $G$  of $F$ satisfy $\langle f_i, g_i \rangle =  \frac{n}{N},\;$ for all $1 \leq i \leq N. $ Hence, by Proposition\ref{prop3point2}, $G$ is a $1-$erasure spectrally optimal dual of \( F \).
	\end{rem}

    \begin{lem}\label{lemma3point5}
    Let $\{a_i\}_{i=1}^s$ be a finite sequence of real numbers such that $\sum\limits_{i=1}^s a_i =0.$ Let $c>0$ be a positive real satisfying $|c+a_i|= |c+a_j|,$ for all $i \neq j.$ Then, $a_i =0,\, 1 \leq i \leq s.$
\end{lem}


 \begin{proof}
Let  \( |C + a_i| = r \) for all \( i \in \{1, \dots, s\} \).  Hence, $ C + a_i = \pm r \quad \text{for each } i.$ This implies that for each \( i \), \( a_i \in \{r - C, -r - C\} \). Let us define $A = r - C, \quad B = -r - C,$ and suppose that exactly \( m \) of the \( a_i \)'s equal \( A \), and the remaining \( s - m \) equal \( B \). Then we have $\sum_{i=1}^s a_i = mA + (s - m)B = 0.$
Substituting \( A \) and \( B \), we get $(2m - s)r - sC = 0.$ But since \( r = |C + a_i| > 0 \) and \( C > 0 \), this equation only holds when \( 2m - s \neq 0 \). In that case, $r = \frac{sC}{2m - s}.$ So for \( r \) to be positive, \( 2m - s \) must be positive. Now suppose \( m = s \), i.e., all \( a_i = A = r - C \). Then $0=\sum_{i=1}^s a_i = s(r - C).$
This leads to, \( r = C \Rightarrow a_i = 0 \) for all \( i \). Similarly, if all \( a_i = B = -r - C \), then $0=\sum_{i=1}^s a_i = s(-r - C) = -s(r + C),$ which is impossible, since $s,r,C>0.$ Hence, the only possible way \( \sum_{i=1}^s a_i = 0 \) under the assumption \( |C + a_i| = \text{constant} \) is when \( a_i = 0 \) for all \( i \in \{1, \dots, s\} \). This completes the proof.
\end{proof}

The following theorem provides a sufficient condition under which the canonical dual frame is guaranteed to be the unique spectrally optimal dual for a one erasure. Let us define $L := \max \left\{ \left\| S_F^{-1/2} f_i \right\| : 1 \leq i \leq N \right\},\,I_1 := \left\{ i \in \{1, \dots, N\} : \left\| S_F^{-1/2} f_i \right\| = L \right\}, \quad I_2 := \{1, 2, \dots, N\} \setminus I_1$ and $H_j := \operatorname{span} \left\{ f_i : i \in I_j \right\}, \quad \text{for } j = 1, 2.$

\begin{thm}\label{prop4point1}
Let \( F = \{f_i\}_{i=1}^N \) be a frame for a Hilbert space \( \mathcal{H}_n \). Suppose that
\begin{itemize}
    \item[(i)] \( H_1 \cap H_2 = \{0\}, \)
    \item[(ii)] \( \{ f_i : i \in I_2 \} \) is linearly independent.
\end{itemize}
Then the canonical dual frame \( S_F^{-1} F \) is a  $1-$erasure spectrally optimal dual frame of \( F \).
\end{thm}

	 \noindent\textbf{Sketch of the Proof.}  
Let \( G = \{S_F^{-1}f_i + u_i\}_{i=1}^N \) be a $1-$erasure spectrally optimal dual of \( F \). Under the assumptions that \( H_1 \cap H_2 = \{0\} \) and \( \{f_i : i \in I_2\} \) is linearly independent, it follows that \( u_i = 0 \) for all \( i \in I_2 \) and \( \sum_{i \in I_1} \langle f_i, u_i \rangle = \mathrm{tr}(T_{U_1} T_{F_1}^*) = 0. \) Since \( G \) is spectrally optimal, we have
\begin{align}\label{eqnspec3point4}
		&\;\;\;\;\;\;\;\;\;\,	\mathrm{AE}_{\mathfrak{R}}^{(1),p} (F,G) \leq \mathrm{AE}_{\mathfrak{R}}^{(1),p} (F,S_{F}^{-1}F)\nonumber\\ 
		& \implies  \left\{\frac{1}{N} \sum\limits_{i \in I_1} \left|\langle f_i, g_i \rangle \right|^p + \frac{1}{N} \sum\limits_{i \in I_2} \left|\langle f_i, g_i \rangle \right|^p \right\}^{\frac{1}{p}}\leq \frac{1}{N} \left\{ \sum\limits_{i \in I_1} \left|\langle f_i, S_{F}^{-1}f_i \rangle \right|^p +  \sum\limits_{i \in I_2} \left|\langle f_i, S_{F}^{-1}f_i \rangle \right|^p \right\}^{\frac{1}{p}} \nonumber\\
        & \implies \sum\limits_{i \in I_1} \left| L+ \langle f_i, u_i \rangle \right|^p \leq \sum\limits_{i \in I_1} L^p =|I_1|L^p.
	\end{align}

Write \( \langle f_i, u_i \rangle = c_i + i d_i \) for real numbers \( c_i, d_i \). Then \( \sum_{i \in I_1} c_i = \sum_{i \in I_1} d_i = 0 \) from the trace condition above. Inequality \eqref{eqnspec3point4} becomes:
\[
\sum_{i \in I_1} \left( (L + c_i)^2 + d_i^2 \right)^{p/2} \leq |I_1| L^p.
\]
Now by Lemma \ref{lemma2point4}, we have:
\small
\begin{align*}
|I_1| L^p &\geq \sum_{i \in I_1} \left( (L + c_i)^2 + d_i^2 \right)^{p/2} \geq \sum_{i \in I_1} |L + c_i|^p \geq |I_1| \left( \frac{1}{|I_1|} \sum_{i \in I_1} |L + c_i| \right)^p \geq |I_1| \left( \frac{ | \sum_{i \in I_1} (L + c_i) | }{|I_1|} \right)^p = |I_1| L^p.
\end{align*}
Equality throughout implies that \( d_i = 0 \) and \( |L + c_i| \) is constant for all \( i \in I_1 \). Therefore by Lemma\ref{lemma3point5}, $c_i =0,$ for all $i \in I_1.$ Hence, the canonical dual is the unique $1-$erasure spectrally optimal dual of $F.$ \hfill$\blacksquare$\\~\\

 The following  proposition establishes a sufficient condition under which a dual frame that is optimal under the spectral radius measure is also optimal under the numerical radius measure.

\begin{prop}\label{prop6point1}
	Let $F$ be a tight frame for $\mathcal{H}_n.$ If $ S_{F}^{-1}F\in \zeta_{\mathfrak{R}}^{(1),p}(F),$ then $ S_{F}^{-1}F\in \zeta_{\mathcal{N}}^{(1),p}(F).$
	
\end{prop}

\begin{proof}
    Let F be a tight frame with tight bound $A.$ For any dual $G$  of $F,$ we have
\begin{align*}
    \mathrm{AE}_{\mathcal{N}}^{(1),p}(F, S_{F}^{-1}F) &=   \left\{\frac{1}{N} \sum\limits_{i=1}^N\left( \frac{|\langle S_{F}^{-1}f_i,f_i\rangle| + \|S_{F}^{-1}f_i\|\|f_i\|}{2}\right)^p\right\}^{\frac{1}{p}}\\&
    = \left\{\frac{1}{N} \sum\limits_{i=1}^N\dfrac{1}{A^p} \|f_i\|^{2p}\right\}^{\frac{1}{p}} \\&
    = \mathrm{AE}_{\mathfrak{R}}^{(1),p}(F, S_{F}^{-1}F) \\&
    \leq \mathrm{AE}_{\mathfrak{R}}^{(1),p}(F, G) \\&
    \leq \left\{\frac{1}{N} \sum \left|\langle f_i, g_i \rangle \right|^p\right\}^{\frac{1}{p}} \\&
    \leq \left\{\frac{1}{N} \sum\limits_{i=1}^N\left( \frac{|\langle g_i,f_i\rangle| + \|g_i\|\|f_i\|}{2}\right)^p\right\}^{\frac{1}{p}}\\&
    = \mathrm{AE}_{\mathcal{N}}^{(1),p}(F, G).
\end{align*}
 Hence, $ S_{F}^{-1}F\in \zeta_{\mathcal{N}}^{(1),p}(F).$    
  
\end{proof}

By Proposition\ref{prop6point1} and Proposition\ref{prop6point2} we can conclude the following:
\begin{cor}
	Let $F$ be a tight frame for $\mathcal{H}_n.$ If $ S_{F}^{-1}F\in \zeta_{\mathcal{R}}^{(1),p}(F),$ then  $ S_{F}^{-1}F\in \zeta_{\mathfrak{F}}^{(1),p}(F).$ 
	
\end{cor}

\noindent The following proposition demonstrates a optimality criteria for the case of uniform Parseval frames.

\begin{prop}
Let \( F = \{f_i\}_{i=1}^N \) be a uniform Parseval frame for a finite-dimensional Hilbert space \( \mathcal{H}_n \). Then the canonical dual \(  S_F^{-1}F = F \in \zeta_{\mathfrak{F}}^{(1),p}(F) \cap \zeta_{\mathfrak{R}}^{(1),p}(F) \cap \zeta_{\mathcal{N}}^{(1),p}(F)\)  and
\(
\mathrm{AE}_{\mathfrak{F}}^{(1),p}(F) = \mathrm{AE}_{\mathfrak{R}}^{(1),p}(F) = \mathrm{AE}_{\mathcal{N}}^{(1),p}(F) = \frac{n}{N}.
\)
\end{prop}

\begin{proof}
Since \( F \) is a uniform Parseval frame, we have \( \|f_i\| = \sqrt{n/N} \) for all \( i \). By definition,$\mathrm{AE}_{\mathfrak{F}}^{(1),p}(F,F) = \left( \frac{1}{N} \sum_{i=1}^N \|f_i\|^{2p} \right)^{1/p} = \left( \frac{1}{N} \cdot N \cdot \left( \frac{n}{N} \right)^p \right)^{1/p} = \frac{n}{N}.$ Hence, \( F \in \zeta_{\mathfrak{F}}^{(1),p}(F) \), by \eqref{eqn4.6}. Similarly we can prove that $ S_F^{-1}F \in \zeta_{\mathfrak{R}}^{(1),p}(F)$ and $ S_F^{-1}F \in \zeta_{\mathcal{N}}^{(1),p}(F)$
Therefore, \( S_F^{-1}F = F \in \zeta_{\mathfrak{F}}^{(1),p}(F) \cap \zeta_{\mathfrak{R}}^{(1),p}(F) \cap \zeta_1^{(p)}(F) \).

\noindent Since the canonical dual achieves the minimum in all three cases, the result follows.
\end{proof}

\begin{example}
Let $F = \{f_1, f_2, f_3\}$ be a frame for $\mathbb{C}^2$, where
\(
f_1 = \begin{bmatrix} 1 \\ 0 \end{bmatrix}, \quad
f_2 = \begin{bmatrix} 0 \\ 1 \end{bmatrix}, \quad
f_3 = \begin{bmatrix} 1 \\ 1 \end{bmatrix}.
\)
\end{example}

\noindent
The frame operator $S_F = \sum_{i=1}^3 f_i f_i^*$ for $F$ is given by $S_F = \begin{bmatrix} 2 & 1 \\ 1 & 2 \end{bmatrix}, \quad \text{so} \quad S_F^{-1} = \frac{1}{3} \begin{bmatrix} 2 & -1 \\ -1 & 2 \end{bmatrix}.$
Therefore, the canonical dual of $F$ is \(S_F^{-1}F = \left\{
\begin{bmatrix} \frac{2}{3} \\[2pt] -\frac{1}{3} \end{bmatrix},\;
\begin{bmatrix} -\frac{1}{3} \\[2pt] \frac{2}{3} \end{bmatrix},\;
\begin{bmatrix} \frac{1}{3} \\[2pt] \frac{1}{3} \end{bmatrix}
\right\}. \)  Any dual frame $G = \{g_i\}_{i=1}^3$ of $F$ can be written as:
\(
G = \left\{
\begin{bmatrix} \frac{2}{3} \\[2pt] -\frac{1}{3} + \alpha \end{bmatrix},\;
\begin{bmatrix} -\frac{1}{3} + \alpha \\[2pt] \frac{2}{3} \end{bmatrix},\;
\begin{bmatrix} \frac{1}{3} - \alpha \\[2pt] \frac{1}{3} - \alpha \end{bmatrix}
\right\}, \quad \alpha \in \mathbb{C}.
\) Using the definition of spectral error under one erasure, we compute: $$\mathrm{AE}_{\mathfrak{R}}^{(1),p}(F, S_F^{-1}F) = \left( \frac{1}{3} \sum_{i=1}^3 \left| \langle f_i, S_F^{-1}f_i \rangle \right|^p \right)^{1/p} = \frac{2}{3}.$$
Hence, $S_F^{-1}F \in \zeta_{\mathfrak{R}}^{(1),p}(F).$

Now, consider the Frobenius-norm based error. One can calculate:
\[
\mathrm{AE}_{\mathfrak{F}}^{(1),2}(F, S_F^{-1}F) = \left( \frac{1}{3} \sum_{i=1}^3 \|f_i\|^2 \|S_F^{-1} f_i\|^2 \right)^{1/2} = \sqrt{\frac{14}{27}} \approx 0.72.
\]
However, taking $\alpha = 0.05$ in the above parametrized dual $G$, we get $\mathrm{AE}_{\mathfrak{F}}^{(1),2}(F, G) \approx 0.435 .$ Therefore, $S_F^{-1}F \notin \zeta_{\mathfrak{F}}^{(1),2}(F).$

\noindent
This example clearly shows that the canonical dual need not be Frobenius-optimal for a non-tight frame, even though it remains optimal under the spectral radius criterion.

         \vspace{1cm}

	\noindent
	{\bf Acknowledgments:} 
	The authors are grateful to the Mohapatra Family Foundation and the College of Graduate Studies of the University of Central Florida for their support during this research. This work is also partially supported by NSF grant DMS-2105038.\\
    
    \noindent
	{\bf Data availability: } 
	The authors declare that no data has been used for this research.

\noindent
	{\bf Ethical approval: }  Not applicable.\\
\noindent
	{\bf Informed consent: } Not applicable.\\
\noindent
	{\bf Conflict of Interest:} Not Applicable.

		 \bibliographystyle{amsplain}

\end{document}